\documentclass[a4paper,12pt]{article}
\usepackage{amsmath,amssymb,enumerate,amsthm,bbm}
\usepackage{color}
\newcommand{\linktojournal}[1]{\relax}
\usepackage[whole]{bxcjkjatype}

\usepackage{hyperref}
\hypersetup{
setpagesize=false,
 bookmarksnumbered=true,%
 bookmarksopen=true,%
 colorlinks=true,%
 linkcolor=blue,
 citecolor=red
}

\newcommand{\R}{\mathbb{R}}

\newcommand{\N}{\mathbb{N}}
\newcommand{\ep}{\varepsilon}
\newcommand{\pa}{\partial}

\newtheorem{theorem}{Theorem}[section]
\newtheorem{lemma}[theorem]{Lemma}
\newtheorem{proposition}[theorem]{Proposition}
\newtheorem{corollary}[theorem]{Corollary}
\theoremstyle{remark}
\newtheorem{remark}{Remark}[section]
\theoremstyle{definition}

\newtheorem{definition}{Definition}[section]

\numberwithin{equation}{section}

\newtheorem*{assumption}{Assumptions}

\makeatletter
\def\@cite#1#2{[{{\bfseries #1}\if@tempswa , #2\fi}]}
\makeatother                  

\begin{document}
\addtolength{\baselineskip}{0.2pt}
\begin{center}

{\Large{\bf 
Appearance of Strauss-type exponent 
in semilinear 
\\[1pt]
wave equations 
with time-dependent 
speed of 
\\[4pt]
propagation
}}
\end{center}

\vspace{5pt}

\begin{center}
Motohiro Sobajima%
\footnote{
Department of Mathematics, 
Faculty of Science and Technology, Tokyo University of Science,  
2641 Yamazaki, Noda-shi, Chiba, 278-8510, Japan,  
E-mail:\ {\tt msobajima1984@gmail.com}}, 
Kimitoshi Tsutaya%
\footnote{
Graduate School of Science and Technology, Hirosaki University, Hirosaki 036-8561, Japan,  
E-mail:\ {\tt tsutaya@hirosaki-u.ac.jp}}, 
Yuta Wakasugi%
\footnote{
Graduate School of Engineering, Hiroshima University, Higashi-Hiroshima 739-8527, Japan,  
E-mail:\ {\tt wakasugi@hiroshima-u.ac.jp}}
\end{center}

\newenvironment{summary}{\vspace{.5\baselineskip}\begin{list}{}{%
     \setlength{\baselineskip}{0.85\baselineskip}
     \setlength{\topsep}{0pt}
     \setlength{\leftmargin}{12mm}
     \setlength{\rightmargin}{12mm}
     \setlength{\listparindent}{0mm}
     \setlength{\itemindent}{\listparindent}
     \setlength{\parsep}{0pt}
     \item\relax}}{\end{list}\vspace{.5\baselineskip}}
\begin{summary}
{\footnotesize {\bf Abstract.}
In this paper, blowup phenomenon 
for the semilinear wave equation 
with time-dependent speed of propagation 
and scattering damping 
is considered
under the smallness of initial data.
Our result contains small data blowup for sub-Strauss exponent 
for the simplest semilinear wave equation 
and also the one for
semilinear generalized Tricomi equation. 
Key ingredient 
is so-called test function method (developed in Ikeda--Sobajima--Wakasa \cite{ISW2019JDE})
with a certain conservative quantity 
via a special solution with 
the Liouville--Green (or WKB) approximation.  
}
\end{summary}




\section{Introduction}
In this paper, we consider the Cauchy problem of 
semilinear wave equations in $\R^N$ $(N\in \N)$
with time-dependent coefficients of the following form:
\begin{equation}\label{intro:Pab}
\begin{cases}
\pa_t^2u-a(t)^2\Delta u + b(t)\pa_tu=|u|^p
&\text{in}\ \R^N\times (0,T), 
\\
(u,\pa_tu)(0)=(\ep f,\ep g)
&\text{in}\ \R^N, 
\end{cases}
\end{equation}
where 
the pair of coefficients $(a,b)$ belongs to $C^2([0,\infty))\times C^1([0,\infty))$ 
and the exponent of the nonlinearity is superlinear, that is, $p>1$.
In particular, the coefficient $a(t)$
for the Laplacian is assumed to be positive to prevent the degeneracy
and $b(t)$ is nonnegative. 
The functions $f$ and $g$ are 
assumed to be compactly supported and smooth,
and the parameter $\ep>0$ describes 
the smallness of initial data. 

The problem \eqref{intro:Pab} is motivated from the study of semilinear wave equation in
the spatially flat Friedmann--Lema\^{i}tre--Robertson--Walker (FLRW) spacetime:
\[
    \partial_t^2 u - \frac{1}{S(t)^2} \Delta u + \frac{NS'(t)}{S(t)} \partial_t u = |u|^p,
\]
where $S(t)$ is the scale factor describing
the expansion or contraction of the spatial metric
(see \cite{TsutayaWakasugi2020} and the references therein for more detail).
Moreover, the problem \eqref{intro:Pab} is related to the semilinear generalized Tricomi equation
\[
    \partial_t^2 u - (1+t)^{2\ell} \Delta u = |u|^p,
\]
where $\ell > -1$.
The equation \eqref{intro:Pab} can be considered as
a generalization of these models.

The aim of the present paper is 
to find blowup solutions with 
arbitrary small initial data
in view of wave concentration phenomena 
near the corresponding light-cone $\mathcal{C}=\{(x,t)\in \R^{N}\times [0,\infty)\,;\,|x|=A(t)\}$
with $A(t)=\int_0^ta(r)\,dr$ $(t\geq 0)$
when $b(t)$ is not identically zero but negligible in a suitable sense.
Our interest is how the conservative quantity
for the corresponding linear problem 
\begin{equation}\label{intro:Pab-lin}
\begin{cases}
\pa_t^2u-a(t)^2\Delta u + b(t)\pa_tu=0
&\text{in}\ \R^N\times (0,\infty), 
\\
(u,\pa_tu)(0)=(f,g)
&\text{in}\ \R^N
\end{cases}
\end{equation}
affects the blowup phenomenon of the semilinear problem \eqref{intro:Pab}.
In particular, under a specific behavior of
$a(t)$ and $b(t)$,
we give a candidate of
the critical exponent
$p_{\rm crit}$ for the problem \eqref{intro:Pab}, which is the threshold value of the power $p$ for existence/non-existence of global-in-time solutions.

The simplest case ($a\equiv 1$ and $b\equiv 0$) 
assigns the well-studied model of the semilinear wave equation of type 
\begin{equation}\label{intro:NW}
\begin{cases}
\pa_t^2u-\Delta u=|u|^p
&\text{in}\ \R^N\times (0,T), 
\\
(u,\pa_tu)(0)=(\ep f,\ep g)
&\text{in}\ \R^N.
\end{cases}
\end{equation}
After the pioneering work by John \cite{John1979}
dealing with the three-dimensional case, 
nowadays it is well known that for $1<p\leq p_{\rm S}(N)$ with 
\begin{align*}
p_{\rm S}(\nu)=\sup\{p>1\;;\;\gamma(\nu,p)>0\}, \quad \gamma_{\rm S}(\nu,p)=2
+(\nu+1)p-(\nu-1)p^2, 
\quad(\nu\geq 1)
\end{align*} 
the problem \eqref{intro:NW} does not possess non-trivial global-in-time solutions. 
On the contrary, if $p>p_{\rm S}(N)$ (with some restriction), then one can find global-in-time solutions
(e.g.,
John \cite{John1979}, Strauss \cite{Strauss1981},
Georgiev--Lindblad--Sogge \cite{GeorgievLindbladSogge1995},
Yordanov--Zhang \cite{YordanovZhang2006},
Zhou \cite{Zhou2007}
and the references therein).

The case of the damped wave equation ($a\equiv b\equiv 1$) also has 
many previous works for existence/non-existence of global-in-time solutions 
with the critical exponent
$p_{\rm crit}=p_{\rm F}(N) := 1+\frac{2}{N}$.
However,
such a case is out of our scope of the present paper
and then
we just refer the reader to
Matsumura \cite{Matsumura1976},
Todorova--Yordanov \cite{TodorovaYordanov2001},
and
Zhang \cite{Zhang2001}.

The next interesting case would be $a\equiv 1$ and 
the time-dependent damping coefficient $b(t)=\mu(1+t)^{-\beta}$.
The corresponding linear problem \eqref{intro:Pab-lin} 
has been systematically studied in the series of papers 
by Wirth
\cite{Wirth2004, Wirth2006, Wirth2007, Wirth2007ADE}
(see also Wirth \cite{WirthThesis}). 
Roughly speaking, 
solutions of \eqref{intro:Pab-lin} is similar (in some sense)
to the ones for the linear wave equation when $\beta>1$. 
In another case $\beta<1$, 
the solutions of \eqref{intro:Pab-lin} have somehow completely different asymptotics.
In the other case $\beta=1$,
the equation \eqref{intro:Pab-lin} is sometimes called the Euler--Poisson--Darboux equation and it has the ``scale-invariance'' 
which provides various (complicated) phenomena.
In the last decade, 
the semilinear problem \eqref{intro:Pab} with $a \equiv 1$ are intensively studied 
(c.f.,
Lin--Nishihara--Zhai \cite{LinNishiharaZhai2012},
Wakasugi \cite{Wakasugi2014},
D'Abbicco \cite{D'Abbicco2015, D'Abbicco2021},
D'Abbicco--Lucente--Reissig \cite{DLR2015},
Lai--Takamura \cite{LaiTakamura2018},
Lai--Takamura--Wakasa \cite{LaiTakamuraWakasa2017}, 
Ikeda--Sobajima \cite{IkedaSobajima2018}, 
Tu--Lin \cite{TuLin2017arXiv, TuLin2019},
Wakasa--Yordanov \cite{WakasaYordanov2019}). 
Roughly speaking, if
$b(t) = \mu (1+t)^{-\beta}$
with $\mu \ge 0$ and $\beta>1$, then (as is expected) the blowup phemomenon 
is almost the same as the case of 
the classical semilinear wave equation \eqref{intro:NW}.
In the case $\beta=1$,
namely the case of semilinear Euler--Poisson--Darboux equation, 
although some of open questions remain, 
it can be expected that the critical exponent 
is given by
$p_{\rm crit}=
\max\{ p_{\rm F}(N), p_{\rm S}(N+\mu) \}$;
note that the effect of the damping term appears in the critical exponent.

In the case
$a(t)= t^{\alpha}$ or $(1+t)^{\alpha}$ with $\alpha > 0$
and $b \equiv 0$,
the problem \eqref{intro:Pab} is called
the semilinear generalized Tricomi equation.
For $\alpha \in \N$,
He--Witt--Yin \cite{HeWittYin2017CVPDE, HeWittYin2017JDE, HeWittYin2020, HeWittYin2021}
found the critical exponent
$p_{\rm crit} = p_{\rm HWY}(N,\alpha)$,
where
\begin{align}
\label{intro:pHWY}
    p_{\rm HWY}(\nu,\alpha)
    &= \sup \{ p> 1 \ ; \  \gamma(\nu,\alpha;p) > 0 \},\\
\label{intro:gammaHWY}
    \gamma(\nu,\alpha;p)
    &=
    2+\Big(\nu+1-\frac{3\alpha}{1+\alpha}\Big)
    p
    -\Big(\nu-1+\frac{\alpha}{1+\alpha}\Big)
    p^2,
\end{align}
and they solved the problem
except for the critical case $p=p_{\rm HWY}(N,\alpha)$.
Later on, for general $\alpha > 0$,
Tu--Lin \cite{TuLin2019arXiv}
proved that the local weak solution must blow up in a finite time and
obtained the estimates of lifespan in
the subcritical and critical cases
$p \le p_{\rm HWY}(N,\alpha)$.
Note that
$p_{\rm HWY}(N,0)=p_{\rm S}(N)$ holds.

When $a(t) = (1+t)^{\alpha}$ with $-1 < \alpha \le 0$
and $b(t) = \mu(1+t)^{-1}$ with $\mu \ge 0$, the problem \eqref{intro:Pab} becomes a generalization of
the semilinear wave equation in the FLRW spacetime.
Tsutaya--Wakasugi \cite{TsutayaWakasugi2020, TsutayaWakasugi2021, TsutayaWakasugi2022}
and Palmieri \cite{Palmieri2021,Palmieri2023}
obtained blowup solutions for
\[
    1 < p \le p_{\rm crit}
    = \max \left\{ p_{\rm F}(N(1+\alpha)), p_{\rm HWY}\left( N+\frac{\mu}{1+\alpha},\alpha \right) \right\}.
\]
We remark that if
$\alpha=0$, then $p_{\rm crit}$ in the above formula reduces to
$\max\{ p_{\rm F}(N), p_{\rm S}(N+\mu) \}$,
and the result of the
semilinear Euler--Poisson--Darboux equation is also included;
if $\mu = 0$, then $p_{\rm crit}$ above coincides with
$p_{\rm HWY}(N,\alpha)$.
We also refer 
to a recent work 
\cite{LaiPalmieriTakamura_preprint}
about a similar blowup 
problem for a generalized model involving time-dependent mass term.

We shall move to 
more general cases of the coefficients.
For the case $a\equiv 1$ and 
$b\in L^1((0,\infty))$, 
in Wakasa--Yordanov \cite{WakasaYordanov2019}
they have succeeded in proving almost the same result 
as in the usual semilinear wave equation \eqref{intro:NW} 
including the critical case $p=p_{\rm S}(N)$ when $N\geq 2$. 
Related to the FLRW model, very recently, Tsutaya--Wakasugi \cite{TsutayaWakasugi_toappear} studied the problem \eqref{intro:Pab} with
general propagation speed $a(t)$ and the scattering damping coefficient
$b \in L^1((0,\infty))$.
They proved the blowup of solutions under the assumption
\begin{equation}\label{intro:TsutayaWakasugi}
    \begin{cases}
    p+1-N \sigma(p-1)>0 & \text { if } \int g(x) d x>0,
\\
    2-N \sigma(p-1)>0 & \text { if } \int g(x) d x=0,
    \end{cases}
\end{equation}
where $\sigma = \limsup_{t\to \infty} (\log A(t)/ \log t )$ with $A(t) = \int_0^t a(r) \,dr$.
In particular, when $\sigma = 0$,
the blowup occurs for any $p>1$.
However, for general $\sigma > 0$ (even for the Minkowski spacetime case $\sigma=1$),
the above condition is not optimal.
One of the main purposes of this paper is to improve the above blowup conditions and
to give the expected critical exponent also for the case $\sigma>0$.

In the present paper, we give 
an argument to discover 
the blowup phenomenon
for a class of general semilinear wave equations
with a time-dependent coefficient of the Laplacian. 
The crucial idea is to use a certain conservative quantity 
for the corresponding linear problem \eqref{intro:Pab-lin},
which will be explained later.

To state the main result of the present paper, 
we give the assumptions
for the coefficients $a$ and $b$.
\begin{assumption}
\begin{itemize}
\item[]
\item[]\hspace{-22pt}
{\bf (A1)}\ $0<a\in C^2([0,\infty))$ and $0\leq b\in C^1([0,\infty))\cap L^1((0,\infty))$,
\item[]\hspace{-22pt}
{\bf (A2)}
$\frac{d^2}{dt^2}(a^{-1}),\frac{d}{dt}(a^{-1}b)\in L^1((0,\infty))$, 
$\frac{d}{dt}(a^{-\frac{1}{2}})\in L^2((0,\infty))$, 
\item[]\hspace{-22pt}
{\bf (A3)}
$\lim\limits_{t\to \infty}\frac{d}{dt}(a^{-1})=\lim\limits_{t\to \infty}(a^{-1}b)=0$.
\end{itemize}
\end{assumption}
We also give the precise definition of solutions to \eqref{intro:Pab}.

\begin{definition}\label{def:sol}
A function $u$ is called a {\it strong solution} of \eqref{intro:Pab} in $(0,T)$
if $u$ belongs to 
\[
C([0,T);H^2(\R^N))
\cap 
C^1([0,T);H^1(\R^N))
\cap 
C^2([0,T);L^2(\R^N))\cap C([0,T);L^{2p}(\R^N))
\] 
and satisfies 
$\pa_t^2u-a(t)^2\Delta u+b(t)\pa_tu=|u|^p$ 
in 
the meaning of $C([0,T);L^2(\R^N))$
together with $u(0)=\ep f$ and $\pa_tu(0)=\ep g$.

The lifespan $T_{\ep}(f,g)$ of such a solution $u$ 
is defined as the maximal existence time $T$ of solutions, that is, 
\[
T_{\ep}(f,g)=\sup\{T>0\;;\;\exists\,u\text{ s.t. there exists a strong solution of \eqref{intro:Pab} in $(0,T)$}\}.
\]
\end{definition}

It is not difficult to prove existence and uniqueness of local-in-time strong solutions to 
\eqref{intro:Pab}. 
For the reader's convenience, in Appendix we briefly give a proof 
through 
the Liouville transform explained below.
We state the expected assertion for existence and uniqueness of 
local-in-time strong solutions to the problem \eqref{intro:Pab} is as follows; 
note that the restriction on $p$ comes from the validity of $|f|^p\in H^{1}(\R^N)$ for $f\in H^2(\R^N)$.
\begin{proposition}
Assume that $a$ and $b$ satisfy {\bf (A1)} and 
\begin{equation}\label{intro:p:existence}
p\in \begin{cases}
(1,\infty) 
& \text{if}\ 1\leq N\leq 4, 
\\
(1,\frac{N-2}{N-4}]& \text{if}\ N\geq 5.
\end{cases}
\end{equation}
Then for every $(f,g)\in H^2(\R^N)\times H^1(\R^N)$ and $\ep\in (0,1]$, 
one has $T_\ep=T_\ep(f,g)>0$. 

Moreover, if 
${\rm supp}\,f\cup {\rm supp}\,g\subset B(0,r_0)$, 
then the corresponding strong solution $u$ of \eqref{intro:Pab} satisfies 
\begin{equation}\label{intro:FSP}
{\rm supp}\,u(t)\subset B(0,r_0+A(t)), \quad t\in (0,T_\ep),
\end{equation}
where $B(0,r)=\{x\in\R^N;|x|<r\}$.
\end{proposition}

We are in a position to state the main result of the present paper.
\begin{theorem}\label{thm:main}
Assume that $a$ and $b$ satisfy {\bf (A1)} 
and \eqref{intro:p:existence}.
Let the pair $(f,g)\in H^2(\R^N)\times H^1(\R^N)$
satisfy 
\begin{equation}\label{ass:ini}
f,g\geq 0,\quad f+g\not\equiv 0,\quad {\rm supp}\,f\cup {\rm supp}\,g\subset B(0,r_0).
\end{equation}
Then there exist positive constants $\ep_0$ and $C$ such that 
the corresponding lifespan $T_\ep=T_\ep(f,g)$ 
of the strong solution of \eqref{intro:Pab} satisfies for every $\ep\in (0,\ep_0]$, 
\begin{equation}\label{eq:thm-eq1}
M_g\ep+M_f^p\ep^p \leq C T_\ep^{-\frac{p+1}{p-1}}A(T_\ep)^{N},
\end{equation}
where $M_f=\int_{\R^N}f\,dx$ and $M_g=\int_{\R^N}g\,dx$.
Moreover, if $a$ and $b$ also satisfy 
{\bf (A2)} and {\bf (A3)}, then
\[
\ep
\leq 
C	T_{\ep}^{-\frac{p^2+1}{p(p-1)}}A(T_\ep)^\frac{N}{p}
\left(\int_{0}^{T_\ep} a(t)^{\frac{p'}{2}}[1+A(t)]^{N-1-\frac{N-1}{2}p'}\,dt
\right)^{\frac{1}{p'}},
\]
where $p'=\frac{p}{p-1}$ is the H\"older conjugate of $p$.
\end{theorem}
\begin{remark}
If 
\[
T^{-\frac{p^2+1}{p(p-1)}}A(T)^\frac{N}{p}
\left(\int_{0}^T a(t)^{\frac{p'}{2}}[1+A(t)]^{N-1-\frac{N-1}{2}p'}\,dt
\right)^{\frac{1}{p'}}
\to 0	\quad \text{as}\ T\to \infty,
\]
then Theorem \ref{thm:main} provides a blowup phenomenon for arbitrary small initial data.
\end{remark}
If 
$a$ and $b$ behave like power type functions at $t\to \infty$, that is, 
\[
a(t)\sim t^{\alpha}, \quad b(t)\sim 
t^{-\beta}\quad\text{as}\ t\to \infty
\]
($\alpha,\beta\in\R$) in a suitable sense, the assertion can be reduced to the following 
corollary
which is quite similar to the case of
semilinear generalized Tricomi equations
discussed in \cite{HeWittYin2017CVPDE, HeWittYin2017JDE, HeWittYin2020, HeWittYin2021}. 
Recall that 
the corresponding critical exponent 
$p_{\rm HWY}(N,\alpha)$ is given in \eqref{intro:pHWY}.

%
\begin{corollary}\label{intro:cor}
Assume that 
there exist constants $\alpha>-1$, $\beta>1$ and $K\geq 1$ such that 
the coefficients $a$ and $b$ satisfy 
{\bf (A1)} with 
\begin{gather*}
K^{-1}\leq (1+t)^{-\alpha}a(t) \leq K, 
\\
(1+t)^{1-\alpha}|a'(t)|+(1+t)^{2-\alpha}|a''(t)|\leq K,
\\
(1+t)^{\beta}b(t)+(1+t)^{1+\beta}|b'(t)|\leq K
\end{gather*}
for $t\geq0$. If $1<p<p_{\rm HWY}(N, \alpha)$ (or equivalently $\gamma(N,\alpha;p)>0$) 
with \eqref{intro:p:existence}, 
then the lifespan $T_\ep=T_\ep(f,g)$ 
of the strong solution to \eqref{intro:Pab}
with $(f,g)\in H^2(\R^N)\times H^1(\R^N)$ 
satisfying \eqref{ass:ini} has the estimate
\[
T_\ep \leq C \ep ^{-\frac{2p(p-1)}{(1+\alpha)\gamma(N,\alpha;p)}}
\] 
for sufficiently small $\ep>0$ and a positive constant $C$ independent of $\ep$,
where $p_{\rm HWY}(N,\alpha)$ and $\gamma(N,\alpha;p)$ are given in \eqref{intro:pHWY} and \eqref{intro:gammaHWY}, respectively.
\end{corollary}
\begin{remark}
We give several remarks on Corollary \ref{intro:cor}.
\begin{itemize}
\item[(i)]
In the simplest case
$a \equiv 1$ and $b \equiv 0$,
we note that 
$\gamma(N,0;p)$ and $p_{\rm HWY}(N,0)$
coincide with $\gamma_{\rm S}(N,p)$ 
and $p_{\rm S}(N)$, respectively.
\item[(ii)]
The function $\gamma(N,\alpha;p)$
and the exponent $p_{\rm HWY}(N,\alpha)$
in Corollary \ref{intro:cor} 
appear in the analysis of blowup phenomena for
semilinear generalized Tricomi equations.
Our result slightly extends the subcritical blowup part of He--Witt-Yin
\cite{HeWittYin2017CVPDE, HeWittYin2017JDE, HeWittYin2020, HeWittYin2021}
and Tu--Lin \cite{TuLin2019arXiv}
to include general propagation speed $a(t)$ behaving like $(1+t)^{\alpha}$ and the scattering damping term.
\item[(iii)]
The case $a\equiv 1$ and 
$b(t)=(1+t)^{-\beta}$ is dealt with 
in Lai--Takamura \cite{LaiTakamura2018}.
The case $a\equiv 1$ and an integrable damping $b$ is treated in 
Wakasa--Yordanov \cite{WakasaYordanov2019} including the critical case $p=p_{\rm S}(N)$. 
Although the assumption of the damping $b$ is slightly stronger than that of \cite{WakasaYordanov2019}, 
our contribution provides 
the case of the time-dependent coefficient $a$ of the Laplacian.
\item[(iv)]
Under the assumption in Corollary \ref{intro:cor}, 
a direct computation shows that 
$\sigma= \limsup_{t\to \infty} (\log A(t)/ \log t )=1+\alpha$. 
Corollary \ref{intro:cor} improves 
the blowup condition \eqref{intro:TsutayaWakasugi}
by Tsutaya--Wakasugi \cite{TsutayaWakasugi_toappear}.
Note that in the typical case $\sigma=1$ (or equivalently, $\alpha = 0$), the condition \eqref{intro:TsutayaWakasugi} when
$M_g >0$
gives the Kato exponent $p_{\rm K}(N)=\frac{N+1}{N-1}$,
and our result provides the Strauss exponent
$p_{\rm S}(N)$.

\item[(v)]
In the critical case $\gamma(N,\alpha;p)=0$ (or equivalently $p=p_{\rm HWY}(N,\alpha)$),
it can be expected that 
the lifespan $T_\ep$ of the strong solution $u$ {\color{blue}has}
the upper bound $T_\ep \leq \exp(C\ep^{-p(p-1)})$ for sufficiently small $\ep>0$.
Our approach will still work in the critical case,
however, it is easy to imagine that this requires an involved technique with long computations.
To concentrate on explaining
the new idea of our approach,
we do not treat the critical case in this paper.
We discuss such a critical phenomenon in a forthcoming paper. 
\end{itemize}
\end{remark}

Here we give a brief explanation of our viewpoint 
with conservative quantities for the corresponding linear equation, 
which is closely related to 
the framework in Ikeda--Sobajima--Wakasa \cite{ISW2019JDE}. 
To explain the idea, 
we rewrite the corresponding linear equation 
of the problem \eqref{intro:Pab} as follows:
\begin{equation}
\label{intro:Pablinear}
\begin{cases}
\pa_t(e^{B(t)}\pa_t v)-e^{B(t)}a(t)^2\Delta v=0
&\text{in}\ \R^N\times (0,\infty),
\\
(v,\pa_tv)(0)=(v_0,v_1)
&\text{in}\ \R^N,
\end{cases}
\end{equation}
where $B(t)=\int_0^tb(r)\,dr$. 
Then one can observe that for each (typical) solution $\Phi$ 
of the corresponding linear equation \eqref{intro:Pablinear}, 
the following formal conservation law is valid: 
\begin{align*}
\frac{d}{dt}\left(
e^{B(t)}
\int_{\R^N} \pa_t v\Phi- v\pa_t\Phi\,dx
\right)
&=
\int_{\R^N} \pa_t(e^{B(t)}\pa_tv)\Phi- v\pa_t(e^{B(t)}\pa_t\Phi)\,dx
\\
&=
e^{B(t)}a(t)^2\int_{\R^N} \Delta v\Phi- v\Delta \Phi\,dx
\\
&=0,
\end{align*}
namely, 
the 
the integral of $e^{B(t)}(\pa_t v\Phi- v\pa_t\Phi)$
is independent of $t$. 
This consideration suggests that various conservative quantities
for the corresponding linear solution of
\eqref{intro:Pablinear}
can be provided via prescribed solutions. 
The simplest solution 
of \eqref{intro:Pablinear}
is of course the constant steady states $\Phi\equiv 1$ 
which provides the conservative quantity
\[
e^{B(t)}\int_{\R^N} \pa_t v\,dx= \int_{\R^N}v_1\,dx, \quad t>0.
\]
On the one hand, the solution via separation of variables 
$\Phi(x,t)=m(t)\phi(x)$ has a parameter $\lambda\in\R$ satisfying 
\[
\lambda=\frac{m''(t)+b(t)m'(t)}{a(t)^2m(t)}=\frac{\Delta \phi(x)}{\phi(x)}.
\]
Then it turns out that the choice 
$\phi(x)=\int_{S^{N-1}}e^{x\cdot \omega}\,dS(\omega)$
in Yordanov--Zhang \cite{YordanovZhang2006}
is reasonable because this function $\phi$ 
satisfies $\Delta\phi=\phi>0$; 
note that the solution $\Phi(x,t)=e^{-t}\phi(x)$ of the linear classical wave equation 
is used in \cite{YordanovZhang2006}. 
Besides, the function $m$ 
is required to be a (positive) decreasing solution of 
\begin{equation}\label{intro:eq:m}
m''(t)+b(t)m'(t)=a(t)^2m(t), \quad t>0.
\end{equation}
To analyse the asymptotic behavior of solutions to \eqref{intro:eq:m}, 
we employ the framework of 
the Liouville--Green approximation 
(or WKB approximation).
Once we find such a solution $\Phi(x,t)=m(t)\phi(x)$, 
we can proceed the test function method with respect to 
the conservative quantity 
\begin{equation}\label{intro:conservative:lightcone}
e^{B(t)}
\int_{\R^N} \big(m(t)\pa_tv(t)- m'(t)v(t)\big)\phi\,dx
=
\int_{\R^N} \big(m(0)v_1- m'(0)v_0\big)\phi\,dx
\end{equation}
for the corresponding linear equation \eqref{intro:Pablinear}. It is worth noting that 
the solution $\Phi(x,t)=m(t)\phi(x)$ 
behaves like an exponentially decaying function with respect to $A(t)$ 
on any compact set in $\R^N$,
and polynomially decaying function 
in $A(t)$ on the neighbourhood of 
$\mathcal{C}(t)=\{x\in \R^N\;;\;|x|=A(t)\}$.
This means that the quantity \eqref{intro:conservative:lightcone} 
describes a concentration phenomenon 
of the wave near the light-cone $\mathcal{C}(t)$. 
The Strauss-type blowup phenomenon can be observed 
by the use of the quantity \eqref{intro:conservative:lightcone}.

The present paper is organized as follows: 
in Section \ref{sec:prelim}
we state the general theorem 
for second-order ordinary differential equations
of the form $y''=Vy$
to clarify the large time behavior of their solutions, 
which is called Liouville--Green approximation. 
In Section \ref{sec:TFM},
we prove Theorem \ref{thm:main} via a modified version of 
test function methods inspired by Ikeda--Sobajima--Wakasa \cite{ISW2019JDE}. 
The crucial point is that in the construction of solutions to 
the corresponding linear problem via separation of variables, 
we employ the Liouville--Green approximation for the time-variable $t$.
	
\section{Preliminaries}\label{sec:prelim}

We collect 
the fundamental 
properties of solutions to 
the following second-order 
ordinary differential equation:
\begin{equation}\label{eq:2nd-ODE}
y''(t)=V(t)y(t), \quad t>0
\end{equation}
with some well-behaved potential $V(t)$.
Clearly, 
the set of all
solutions of the equation \eqref{eq:2nd-ODE} 
is a
two-dimensional subspace of $C^2(\R_+)$. 
In a special case, 
the asymptotic profiles of its fundamental solutions
are well--known. 
The following lemma is a consequence of the Liouville--Green approximation. 
For details, see Olver \cite[Chapter 6]{OlverBook}
(and also Metafune--Sobajima \cite{MSprocEQUADIFF}).

\begin{lemma}\label{lem:ODEprelim}
Let $V\in C([0,\infty))$. Assume that 
there exist functions $(\varphi,V_*)\in C^2([0,\infty))\times C([0,\infty))$ such that 
$V(t)=\varphi(t)^2+V_*(t)$ with 
\begin{itemize}
\item[\bf (i)] $\varphi(t)>0$ and $\varphi\notin L^1((0,\infty))$,
\item[\bf (ii)] $W_{\varphi,V_*}:=
-\varphi^{-\frac{1}{2}}
(\varphi^{-\frac{1}{2}})''+V_*\varphi^{-1}\in L^1((0,\infty))$. 
\end{itemize}
Then there exists a fundamental system 
$(\psi_+,\psi_-)$ of the equation 
\eqref{eq:2nd-ODE} such that 
the following asymptotics hold:
\begin{gather*}
\lim_{t\to \infty}
\left(
\frac{\varphi(t)^{\frac{1}{2}}\psi_{+}(t)}{e^{\Phi(t)}}
\right)
=1, 
\quad 
\lim_{t\to \infty}
\left(
\frac{[\varphi^{\frac{1}{2}}\psi_{+}]'(t)}{\varphi(t)e^{\Phi(t)}}
\right)=1, 
\\
\lim_{t\to \infty}
\left(
\frac{\varphi(t)^{\frac{1}{2}}\psi_{-}(t)}{e^{- \Phi(t)}}
\right)
=1,
\quad 
\lim_{t\to \infty}
\left(
\frac{[\varphi^{\frac{1}{2}}\psi_{-}]'(t)}{\varphi(t)e^{-\Phi(t)}}
\right)=-1,
\end{gather*}
where $\Phi(t)=\int_0^t\varphi(r)\,dr$.
\end{lemma}
\begin{remark}
{(i)} In the proof of 
Lemma \ref{lem:ODEprelim}, 
the Liouville transform 
$y\mapsto \widetilde{y}$ given as 
\[
y(t)=\varphi(t)^{-\frac{1}{2}}\widetilde{y}(\Phi(t))
\]
which translates the equation \eqref{eq:2nd-ODE} into  $\widetilde{y}''=\widetilde{y}+\widetilde{W}$ with a suitable negligible correction $\widetilde{W}$.
This enables us to find the desired asymptotic behavior for $\widetilde{y}$ and then the one for $y$. 
This consideration will be also used in Appendix. 

{(ii)}\ According to 
Lemma \ref{lem:ODEprelim}, 
there exists a pair $(\psi_+,\psi_-)$ of linearly independent solutions 
to \eqref{eq:2nd-ODE} satisfying the asymptotic behavior
\[
\psi_+(t)\sim \varphi(t)^{-\frac{1}{2}}e^{\Phi(t)}, 
\quad
\psi_-(t)\sim \varphi(t)^{-\frac{1}{2}}e^{-\Phi(t)} 
\]
as $t\to \infty$.
In particular, our strategy requires 
the above behavior of $\psi_-$. 
\end{remark}

\section{Test function methods with conservative quantities}
\label{sec:TFM}

In this section we employ the test function method 
with certain 
conservative quantities via special solutions of the corresponding linear equation
\begin{equation}\label{eq:linear}
\pa_t\Phi-a(t)^2\Delta \Phi +b(t)\pa_t\Phi =0\quad\text{in}\ \R^N\times (0,\infty).
\end{equation}
As explained in Introduction, we frequently use 
the description 
\[
\pa_t(e^{B(t)}\pa_t \Phi)=e^{B(t)}a(t)^2\Delta \Phi
\]
with $B(t)=\int_0^tb(r)\,dr$
and the corresponding conservative quantity 
for the linear problem \eqref{intro:Pab-lin}:
\[
Q_{\Phi}(u;t)=e^{B(t)}
\int_{\R^N} \pa_t u\Phi- u\pa_t\Phi\,dx, \quad t\in [0,T_\ep).
\]

The following representation is useful 
in the framework of test function methods.
\begin{lemma}\label{lem:inv}
Let $u$ be the strong solution of \eqref{intro:Pab} in $(0,T_\ep)$
with $(f,g)\in H^2(\R^N)\times H^1(\R^N)$
satisfying \eqref{ass:ini}. 
If a function $\Phi\in C^2(\R^N\times [0,\infty))$ satisfies 
\eqref{eq:linear}, then 
for every $\psi\in C_0^2([0,T))$, 
\begin{align}
\nonumber 
&
-
\frac{d}{dt}\left(
\psi
Q_{\Phi}(u;t)
-
\psi'
e^{B(t)}\int_{\R^N}u\Phi\,dx\right)
+\psi e^{B(t)}\int_{\R^N}|u|^p\Phi\,dx
\\
&=
\label{eq:lem:inv}
2\psi' e^{B(t)}\int_{\R^N}u\pa_t\Phi\,dx
+(\psi' e^{B(t)})' \int_{\R^N}u\Phi\,dx.
\end{align}
\end{lemma}
\begin{proof}
Since the solution $u$ satisfies ${\rm supp}\,u(t)\subset {\color{blue}B}(0,r_0+A(t))$
for $t\in (0,T_\ep)$, 
we can use $\psi \Phi$ as the test function.  
More precisely we use a smooth cut-off function $\widetilde{\psi}$ 
satisfying 
$\widetilde{\psi}\equiv 1$ on $\{(x,t)\;;\;|x|\leq r_0+A(t)\}$
and
$\widetilde{\psi}\equiv 0$ on $\{(x,t)\;;\;|x|\geq 1+r_0+A(t)\}$.  
A direct calculation with integration by parts shows that
\begin{align*}
&\frac{d}{dt}
\left(e^{B(t)}\int_{\R^N}\pa_tu(\psi\Phi)-u\pa_t(\psi\Phi)\,dx\right)
\\
&=
\int_{\R^N}
\pa_t(e^{B(t)}\pa_tu)\psi\Phi
-u\pa_t\big(e^{B(t)}\pa_t(\psi\Phi)\big)\,dx
\\
&=
\psi e^{B(t)}
\int_{\R^N}|u|^p\Phi\,dx
+
\int_{\R^N}u
\Big(e^{B(t)} \psi a(t)^2\Delta \Phi-\pa_t\big(e^{B(t)}\pa_t(\psi\Phi)\big)\Big)\,dx.
\end{align*}
By \eqref{eq:linear}, we have
\begin{align*}
\pa_t(e^{B(t)}\pa_t(\psi\Phi))
&=
\pa_t(e^{B(t)}\psi')\Phi
+2\psi' e^{B(t)}\pa_t\Phi
+\psi \pa_t(e^{B(t)}\pa_t\Phi)
\\
&=
\pa_t(e^{B(t)}\psi')\Phi
+2\psi' e^{B(t)}\pa_t\Phi
+\psi e^{B(t)}a(t)^2\Delta \Phi.
\end{align*}
Combining these equalities, we deduce the desired equality. 
\end{proof}
The above lemma suggests the choice of cut-off functions for the variable $t$ 
in the following way.
Fix a function $\eta\in C^\infty(\R)$ satisfying 
$\mathbbm{1}_{(-\infty,\frac{1}{2})}\leq \eta\leq \mathbbm{1}_{(-\infty,1)}$, 
where $\mathbbm{1}_{K}$ is the indicator function on a set $K$.
Then for the parameter $R>0$, define 
\begin{equation}\label{eq:psi_R}
\psi_R(t)=\eta\left(\frac{B_*(t)}{B_*(R)}\right)^{\frac{2p}{p-1}}, \quad B_*(t)=\int_{0}^te^{-B(r)}\,dr.
\end{equation}
Then by a direct calculation, we have
\begin{lemma}\label{lem:cutoff-estimate}
Assume that $b\in L^1((0,\infty))$. Let $\psi_R$ be as in \eqref{eq:psi_R}. 
Then one has $\psi_R\in C_0^2([0,R))$.
Moreover, there exists a positive constant $C$ independent of $R$ such that 
for every $R>0$ and $t\geq 0$, 
\[
|\psi_R'(t)e^{B(t)}|
\leq  
\frac{C}{R}\psi_R(t)^{\frac{1}{p}},
\quad 
|
(\psi_R'(t)e^{B(t)})'|
\leq  
\frac{C}{R^{2}}\psi_R(t)^{\frac{1}{p}}.
\]
\end{lemma}
\begin{proof}
Set $\zeta=\eta^{2p'}$ for simplicity. 
By the definition of $\psi_R$, we can compute
\begin{align*}
\psi_R'(t)e^{B(t)}
&=\zeta'\left(\frac{B_*(t)}{B_*(R)}\right)\frac{1}{B_*(R)},
\\
(\psi_R'(t)e^{B(t)})'
&=\zeta''\left(\frac{B_*(t)}{B_*(R)}\right)\frac{e^{-B(t)}}{B_*(R)^2}.
\end{align*}
Noting that $e^{-B(+\infty)}t\leq B_*(t)\leq t$, we can find the desired estimates.
\end{proof}

\begin{proposition}\label{prop:TFM1}
Assume that $a$ and $b$ satisfy {\bf (A1)}. Then 
there exist positive constants $\ep_0$ and $C$ independent of $\ep$ such that 
for every $\ep\in (0,\ep_0)$ and $R\in (0,T_\ep)$, 
\[
\ep M_g+\int_0^R \psi_{R}\int_{\R^N}|u|^p\,dx\,dt
\leq CR^{-\frac{p+1}{p-1}}\big(r_0+A(R)\big)^{N}.
\]
\end{proposition}
\begin{proof}
We take $\Phi\equiv1$, that is, we focus our attention to 
the quantity
\[
Q_1(u;t)=e^{B(t)}\int_{\R^N}\pa_tu\,dx, \quad Q_1(u;0)=\ep M_g.
\]
Then the right-hand side of \eqref{eq:lem:inv} in Lemma \ref{lem:inv}
can be estimated as follows:
\begin{align*}
&2e^{B(t)}\psi_R'\int_{\R^N}u\pa_t\Phi\,dx
+\pa_t(e^{B(t)}\pa_t\psi_R) \int_{\R^N}u\Phi\,dx
\\
&\leq 
\frac{C}{R^2}\psi_{R}^{\frac{1}{p}}\int_{\R^N}|u|\,dx
\\
&\leq 
\frac{C}{R^2}
\left(\int_{{\rm supp}\,u(t)}\,dx\right)^{\frac{1}{p'}}
\left(\psi_{R}\int_{\R^N}|u|^p\,dx\right)^{\frac{1}{p}},
\end{align*}
where we have used the finite speed of propagation.
Integrating it over $(0,R)$ with Lemma \ref{lem:inv} 
and using H\"older's and Young's inequalities, we have
\begin{align*}
&
Q_1(u;0)+\int_0^R e^{B(t)}\psi_R\int_{\R^N}|u|^p\,dx\,dt
\\
&\leq 
\frac{C}{R^2}
\left(\int_0^R\int_{{\rm supp}\,u(t)}\,dx\,dt\right)^{\frac{1}{p'}}
\left(\int_0^{R}\psi_{R}\int_{\R^N}|u|^p\,dx\,dt\right)^{\frac{1}{p}}
\\
&\leq 
\frac{C^{p'}}{p'R^{2p'}}\int_0^R\int_{{\rm supp}\,u(t)}\,dx\,dt
	+
\frac{1}{p}\int_0^R \psi_{R}\int_{\R^N}|u|^p\,dx\,dt
\end{align*}
and therefore 
\begin{align*}
\ep M_g+\int_0^R \psi_R\int_{\R^N}|u|^p\,dx\,dt
\leq 
\frac{C^{p'}}{R^{2p'}}\int_0^R\int_{{\rm supp}\,u(t)}\,dx\,dt.
\end{align*}
Computing the right-hand side of the above estimate 
via the finite speed of propagation ${\rm supp}\,u(t)\subset \overline{B}(0,r_0+A(t))$, 
one can find the desired inequality. 
\end{proof}

\begin{proof}[Proof of the first estimate \eqref{eq:thm-eq1} in Theorem \ref{thm:main}]
We can assume that $T_\ep >2e^{B(+\infty)}$. 
In view of Proposition \ref{prop:TFM1}, it suffices to find a lower bound for 
\[
\int_0^R \psi_R\int_{\R^N}|u|^p\,dx\,dt
\]
with $R\in (2e^{B(+\infty)},T_\ep)$.
By the definition of $\psi_R$, we see that 
for $0\leq t\leq 1$, 
one has $\frac{B_*(t)}{B_*(R)}\leq e^{B(+\infty)}\frac{t}{R}\leq \frac{1}{2}$ 
which implies that $\psi_R(t)=1$ for $t\in [0,1]$. 
This yields
\[
\int_0^1 \int_{\R^N}|u|^p\,dx\,dt
\leq \int_0^R \psi_R\int_{\R^N}|u|^p\,dx\,dt.
\] 
Observe again that the solution $u$ is supported in 
$B(0,r_0+A(1))$ for $t\in[0,1]$. 
Multiplying the both sides of the equation in \eqref{intro:Pab} by $e^{B(t)}$
and integrating them over $\R^N$, we have 
\[
e^{B(t)}
\frac{d}{dt}
\int_{\R^N} u(t) \,dx
=
\ep M_g+\int_0^t e^{B(r)}\int_{\R^N}|u(r)|^p\,dx\,dr\geq 0, \quad t\in (0,1),
\]
which yields 
\begin{equation*}
\int_{\R^N} u(t) \,dx\geq 
\ep M_f, \quad t\in [0,1].
\end{equation*}
Therefore using the H\"older inequality, we have
for every $t\in [0,1]$, 
\[
\ep M_f\leq \int_{\R^N} u(t) \,dx\leq 
C(r_0+A(1))^{\frac{N}{p'}}
\left(\int_{\R^N}|u|^p\,dx\right)^{\frac{1}{p}}
\]
for some positive constant $C$. Consequently, 
we deduce 
\begin{align*}
\frac{\ep^p M_f^p}{C^p(r_0+A(1))^{N(p-1)}}
&\leq \int_0^1 \int_{\R^N}|u|^p\,dx\,dt.
\end{align*}
The proof of \eqref{eq:thm-eq1} is complete.
\end{proof}

Next we impose {\bf (A2)} and {\bf (A3)} on $a$ and $b$. 
Then we argue 
the conservative quantity for \eqref{eq:linear} via 
solutions with separation of variables. 
Namely, consider 
solutions of \eqref{eq:linear} of the form
\[
\Phi(x,t)=m(t)\phi(x), \quad (x,t)\in \R^N\times [0,\infty).
\]
As explained in Introduction, we choose 
\[
\phi(x)=\int_{S^{N-1}}e^{\omega\cdot x}\,dS(\omega)
\]
and $m\in C^2([0,\infty))$ satisfying 
the ordinary differential equation
\begin{equation}\label{eq:ode-for-sol}
m''(t)+b(t)m'(t)=a(t)^2m(t), \quad  t>0.
\end{equation}
The following lemma collects 
required properties of the solution $m$ of 
\eqref{eq:ode-for-sol} to 
proceed our test function method.
\begin{lemma}\label{lem:sol-of-ode}
Assume that $a$, $b$ satisfy {\bf (A1)--(A3)}. 
Then there exist a solution $m_*\in C^2([0,\infty))$ of \eqref{eq:ode-for-sol}
and a positive constant $\delta_*>0$ 
such that for every $t\geq 0$, 
\begin{gather*}
\delta_* a(t)^{-\frac{1}{2}}e^{-A(t)}
\leq 
m_*(t)\leq  
\delta_*^{-1}a(t)^{-\frac{1}{2}}e^{-A(t)},
\\
-
\delta_*^{-1}a(t)^{\frac{1}{2}}e^{-A(t)}
\leq 
m_*'(t)\leq  
-\delta_* a(t)^{\frac{1}{2}}e^{-A(t)}.
\end{gather*}
\end{lemma}
\begin{proof}
Under the assumptions {\bf (A1)} and {\bf (A2)}, 
we can justify that the assumption in 
Lemma \ref{lem:ODEprelim}
with $(\varphi,V_*)=(a,\frac{b'}{2}+\frac{b^2}{4})$ is fulfilled. 
Indeed, we see from {\bf (A2)} that 
\begin{align*}
a^{-\frac{1}{2}}\frac{d^2}{dt^2}(a^{-\frac{1}{2}})
=
\frac{1}{2}\frac{d^2}{dt^2}(a^{-1})
-
\left|\frac{d}{dt}(a^{-\frac{1}{2}})\right|^2\in L^1((0,\infty)),
\end{align*}
and also $\frac{d}{dt}(a^{-1}), a^{-1}b\in L^\infty((0,\infty))$.
Therefore we also find 
\begin{gather*}
a^{-1}b'
=
\frac{d}{dt}\big(a^{-1}b\big)-\frac{d}{dt}\big(a^{-1}\big)b\in L^1((0,\infty)), 
\quad 
a^{-1}b^2
=
\big(a^{-1}b\big)b\in L^1((0,\infty)).
\end{gather*}
Applying Lemma \ref{lem:ODEprelim}, we find a function $\psi_-$ satisfying  
\[
\psi_-''=\left(a^2+\frac{b'}{2}+\frac{b^2}{4}\right)\psi_-
\]
with the large time behavior
\begin{align*}
a(t)^{\frac{1}{2}}e^{A(t)}\psi_{-}(t)
\to 1, 
\quad 
a(t)^{-1}e^{A(t)}[a^{\frac{1}{2}}\psi_{-}]'(t)\to -1
\quad\text{as}\ t\to \infty.
\end{align*}
Here we define  
\[
m_*(t)=\kappa_* e^{-\frac{1}{2}B(t)}\psi_-(t),\quad  t\geq0
\]
with $\kappa_*=e^{\frac{1}{2}B(+\infty)}$.
We can easily check that $m_*$ satisfies \eqref{eq:ode-for-sol} 
and 
\[
a(t)^{\frac{1}{2}}e^{A(t)}m_*(t) \to 1\quad\text{as}\ t\to \infty.
\]
On the one hand, 
by virtue of {\bf (A3)}, we have
\begin{align*}
a(t)^{-\frac{1}{2}}e^{A(t)}m_*'(t)
&=
\kappa_* e^{-\frac{1}{2}B(t)}
a(t)^{-\frac{1}{2}}e^{A(t)}
\Big(
\psi_-'(t)-\frac{b(t)}{2}\psi_-(t)
\Big)
\\
&=
\kappa_* e^{-\frac{1}{2}B(t)}
\left[
a(t)^{-1}e^{A(t)}
[a^{\frac{1}{2}}\psi_-]'(t)
-
\frac{1}{2}\Big(\frac{a'(t)}{a(t)^2}+\frac{b(t)}{a(t)}\Big)
a(t)^{\frac{1}{2}}e^{A(t)}
\psi_-(t)\right]
\\
&\to -1\quad\text{as}\ t\to \infty.
\end{align*}
Therefore there exists a positive constant $t_*$ such that 
for every $t\geq t_*$, 
\[
\frac{1}{2}a(t)^{-\frac{1}{2}}e^{-A(t)}
\leq m_*(t)\leq 
2a(t)^{-\frac{1}{2}}e^{-A(t)}, 
\ \ 
-2a(t)^{\frac{1}{2}}e^{-A(t)}
\leq m_*'(t)\leq 
-\frac{1}{2}a(t)^{\frac{1}{2}}e^{-A(t)}. 
\]
For the property in $[0,t_*]$, 
we can see from the equation \eqref{eq:ode-for-sol}, that 
\begin{align*}
\Big(e^{B(t)}m_*'(t)m_*(t)\Big)'
&=\big(e^{B(t)}m_*'(t)\big)'m_*(t)+e^{B(t)}m_*'(t)^2
\\
&=e^{B(t)}a(t)^2m_*(t)^2+e^{B(t)}m_*'(t)^2\geq 0. 
\end{align*}
Integrating it over $[t,t_*]$ we can find for all $t\in [0,t_*]$, 
\[
e^{B(t)}m_*'(t)m_*(t)\leq e^{B(t_*)}m_*'(t_*)m_*(t_*)
\leq - \frac{1}{4}e^{B(t_*)}e^{-2A(t_*)}< 0.
\]
This means that both of the functions $m_*$ and $m_*'$ do not change their sign.
Since the interval $[0,t_*]$ is compact, we can reach the desired estimate 
for $m_*$ and also $m_*'$.
The proof is complete.
\end{proof}

Using the function $m_*$ constructed in Lemma \ref{lem:sol-of-ode}, 
we define 
\[
\Phi_*(x,t)=m_*(t)\phi(x), \quad (x,t)\in \R^N\times (0,\infty)
\]
which is a typical solution of \eqref{eq:linear}. 
Then we focus our attention to the corresponding 
quantity 
\[
Q_*(u;t)=Q_{\Phi_*}(u;t)
=e^{B(t)}
\int_{\R^N} \big(m_*(t)\pa_tu(t)- m_*'(t)u(t)\big)\phi\,dx.
\]
It is worth noticing that 
under the assumption of Theorem \ref{thm:main},  
the initial value of this quantity is positive since
\begin{align*}
Q_*(u;0)
&=\ep \int_{\R^N} \big(m_*(0)g- m_*'(0)f\big)\phi\,dx
\\
&\geq \delta_* \ep 
\left(
a(0)^{-\frac{1}{2}} \int_{\R^N}g \phi\,dx
+
a(0)^{\frac{1}{2}}\int_{\R^N} f\phi\,dx\right)
\\&=\delta\ep >0
\end{align*}
for some $\delta>0$.
\begin{proof}[The rest of the proof of Theorem \ref{thm:main}]
Applying Lemma \ref{lem:inv}, we have
\begin{align*}
&
-
\frac{d}{dt}\left(
\psi_R Q_*(u;t)
-
e^{B(t)}\psi_R'm_*\int_{\R^N}u\phi\,dx\right)
+e^{B(t)}\psi_Rm_*\int_{\R^N}|u|^p\phi\,dx
\\
&=
\Big(
2e^{B(t)}\psi_R'm_*'
+(e^{B(t)}\psi_R')'m_*\Big)\int_{\R^N}u\phi\,dx,
\end{align*}
where $\psi_R$ is given in \eqref{eq:psi_R}.
Integrating it over $[0,R]$
and applying Lemmas \ref{lem:cutoff-estimate} and \ref{lem:sol-of-ode}, 
we have 
\begin{align*}
Q_*(u;0)&\leq Q_*(u;0)+
\int_0^R
e^{B(t)}\psi_Rm_*\int_{\R^N}|u|^p\phi\,dx
\,dt
\\
&=
\int_{{\rm supp}\,\psi_R'}
\Big(
2e^{B(t)}\psi_R'm_*'
+(e^{B(t)}\psi_R')'m_*\Big)\int_{\R^N}u\phi\,dxdt
\\
&\leq 
\frac{2C}{\delta_* R}
\int_{{\rm supp}\,\psi_R'}
\Big(
1+\frac{1}{2R a(t)}\Big)a(t)^{\frac{1}{2}}e^{-A(t)}\psi_R^{\frac{1}{p}}\|u\|_{L^p}\|\phi\|_{L^{p'}({\rm supp}\,u(t))}\,dt.
\end{align*}
Noting that the conditions {\bf (A1)} and {\bf (A3)} also give
the boundedness of 
$\frac{1}{(1+t)a(t)}$, we can see that 
\begin{align*}
\delta \ep 
&\leq 
C'
\int_{{\rm supp}\,\psi_R'}
a(t)^{\frac{1}{2}}e^{-A(t)}\psi_R^{\frac{1}{p}}\|u\|_{L^p}\|\phi\|_{L^{p'}({\rm supp}\,u(t))}\,dt
\\
&\leq 
C'
\left(
\int_{{\rm supp}\,\psi_R'}
a(t)^{\frac{p'}{2}}e^{-p'A(t)}\|\phi\|_{L^{p'}({\rm supp}\,u(t))}^{p'}\,dt\right)^{\frac{1}{p'}}
\left(
\int_0^R
\psi_R \|u\|_{L^p}^p\,dt\right)^{\frac{1}{p}}.
\end{align*}
Combining the estimate 
\[
\|\phi\|_{L^{p'}({\rm supp}\,u(t))}^{p'}\leq 
C''(1+r_0+A(t))^{N-1-\frac{N-1}{2}p'}e^{p'A(t)}
\]
(see e.g., \cite{YordanovZhang2006}) and Proposition \ref{prop:TFM1}, we arrive at 
\[
\ep \leq C'''
R^{-\frac{p^2+1}{p(p-1)}}(r_0+A(R))^{\frac{N}{p}}
\left(
\int_{{\rm supp}\,\psi_R'}
a(t)^{\frac{p'}{2}}(1+r_0+A(t))^{N-1-\frac{N-1}{2}p'}
\,dt\right)^{\frac{1}{p'}}
\]
for some positive constant $C'''$.
\end{proof}

\section*{Appendix}
\renewcommand{\thesection}{A}
\setcounter{theorem}{0}
\setcounter{equation}{0}

We shall give a proof of existence of local-in-time strong solutions 
to \eqref{intro:Pab} from the viewpoint of the Liouville transform.
The statement is the following:
\begin{proposition}\label{prop:appendix}
Assume that $0<a\in C^2([0,\infty))$, $b\in C([0,\infty))$
and $(u_0,u_1)\in H^2(\R^N)\times H^1(\R^N)$. 
If $p$ satisfies \eqref{intro:p:existence},
then there exists a positive constant $T$ such that the problem 
\begin{equation}\label{appendix:problem}
\begin{cases}
\pa_t^2u-a(t)^2\Delta u + b(t)\pa_tu=|u|^p
&\text{in}\ \R^N\times (0,T), 
\\
(u,\pa_tu)(0)=(u_0,u_1)
&\text{in}\ \R^N
\end{cases}
\end{equation}
has a unique strong solution $u$. Additionally, if ${\rm supp}\,u_0\cup 
{\rm supp}\,u_1\subset B(0,r_0)$, then the solution $u$ satisfies  
\[
{\rm supp}\,u_1\subset 
B\left(0,r_0+\int_0^ta(r)\,dr\right).
\]
\end{proposition}

\begin{proof}
To find a local-in-time solution of \eqref{appendix:problem}, 
we employ the Liouville transform, 
that is, we set a new unknown function $w(x,s)$ as
\[
u(x,t)=a(t)^{-\frac{1}{2}}w\big(x,A(t)\big), \quad A(t)=\int_0^ta(r)\,dr, \quad t\in(0,T).
\]
Observing that 
\begin{align*}
\pa_tu(x,t)
&=
a(t)^{\frac{1}{2}}\pa_s w\big(x,A(t)\big)
+(a(t)^{-\frac{1}{2}})'w\big(x,A(t)\big),
\\
\pa_t^2u(x,t)
&=
a(t)^{\frac{3}{2}}\pa_s^2 w\big(x,A(t)\big)
+(a(t)^{-\frac{1}{2}})''w\big(x,A(t)\big),
\\
\Delta u(x,t)
&=
a(t)^{-\frac{1}{2}}\Delta w\big(x,A(t)\big),
\end{align*}
we can see that the problem \eqref{appendix:problem} can be translated into 
the following evolution problem with respect to the variable $s=A(t)$
\begin{equation}\label{appendix:problem-translated}
\begin{cases}
\pa_s^2 w
-
\Delta w
+
\widetilde{b}(s)\pa_s w
+\widetilde{V}(s)w
=
\widetilde{c}(s)|w|^p, \quad \text{in}\ \R^N\times (0,S),
\\
(w,\pa_sw)(0)=(w_0,w_1)=\big(a(0)^{\frac{1}{2}}u_0, a(0)^{-\frac{1}{2}}u_1+\frac{1}{2}a(0)^{-\frac{3}{2}}a'(0)u_0\big),
\end{cases}
\end{equation}
where 
$S>0$ is determined later, and  
 $\widetilde{b}$, $\widetilde{V}$ and $\widetilde{c}$ are 
continuous functions given by 
\begin{align*}
\widetilde{b}(A(t))=\Theta(t)^2b(t), 
\quad
\widetilde{V}(A(t))=\Theta(t)^3\big(
\Theta''(t)+b(t)\Theta'(t)\big), 
\quad
\widetilde{c}(A(t))=
\Theta(t)^{p+3}
\end{align*}
with $\Theta(t)=a(t)^{-\frac{1}{2}}$.
Then setting 
\[
G(s,w,\pa_s w)=-\widetilde{b}(s)\pa_s w-\widetilde{V}(s)w+\widetilde{c}(s)|w|^p,
\]
we find that \eqref{appendix:problem} is equivalent (at least local-in-time) to 
the following semilinear wave equation 
with {\it constant speed of propagation}\/:
\begin{equation}\label{appendix:problem2}
\begin{cases}
\pa_s^2 w-\Delta w=G(s,w,\pa_s w), \quad \text{in}\ \R^N\times (0,S),
\\
(w,\pa_sw)(0)=(w_0,w_1)\in H^2(\R^N)\times H^1(\R^N).
\end{cases}
\end{equation}
Then we use the standard argument with the contraction mapping principle 
in the space $\mathcal{H}=H^1(\R^N)\times L^2(\R^N)$.
Define the linear quasi-$m$-dissipative operator $\mathcal{L}(w,z)=(z,\Delta w)$
in $\mathcal{H}$ endowed with the domain $D(\mathcal{L})=H^2(\R^N)\times H^1(\R^N)$
and 
the nonlinear map 
$\mathcal{G}(s,(w,z))=(0,G(s,w,z))$ which corresponds to the term $G(s,w,\pa_s w)$.
Then 
by virtue of the fact $|f|^p\in H^1(\R^N)$ for $f\in H^2(\R^N)$ 
under \eqref{intro:p:existence}, 
by a straightforward calculation, 
we can choose a sufficiently small $S>0$ such that 
the following map $\Xi$ is well-defined and contractive:
\[
\Xi \mathcal{W}(s)
=
e^{s\mathcal{L}}(w_0,w_1)+\int_0^{s}e^{(s-r)\mathcal{L}}
\mathcal{G}(r,\mathcal{W}(r))\,dr, 
\quad s\in[0,S]
\]
defined in  
\[
\mathcal{X}_S=
\left\{\mathcal{W}\in C([0,S];\mathcal{H})\cap 
L^\infty(0,S;D(\mathcal{L}))
\;;\; \big\|\|\mathcal{W}(\cdot)\|_{D(\mathcal{L})}\big\|_{L^\infty((0,S))}\leq 2\|(w_0,w_1)\|_{D(\mathcal{L})}\right\}
\]
with the metric $d(\mathcal{W}_1,\mathcal{W}_2)=\max\limits_{0\leq s\leq S}\|\mathcal{W}_1-\mathcal{W}_2\|_{\mathcal{H}}$ 
to itself. Therefore we can find the fixed point $\mathcal{W}^\sharp=(w^\sharp,
{\color{blue}z}^\sharp)\in \mathcal{X}_S$ of $\Xi$.
Then we further consider 
the inhomogeneous problem 
\begin{equation*}
\begin{cases}
\pa_s^2 w-\Delta w=G(s,w^\sharp(s),z^\sharp(s)), \quad \text{in}\ \R^N\times (0,S),
\\
(w,\pa_sw)(0)=(w_0,w_1)\in H^2(\R^N)\times H^1(\R^N).
\end{cases}
\end{equation*}
The above problem has the strong solution 
$w\in C^2([0,S];L^2(\R^N))\cap C^1([0,S];H^1(\R^N))\cap C([0,S];H^2(\R^N))$
which satisfies 
\[
(w(s),\pa_sw(s))=
e^{s\mathcal{L}}(w_0,w_1)+\int_0^{s}e^{(s-r)\mathcal{L}}
\mathcal{G}(r,\mathcal{W}^\sharp(r))\,dr
=\Xi\mathcal{W}^\sharp(s)=\mathcal{W}^\sharp(s).
\]
This means that $w^\sharp=w$ is the strong solution of the problem \eqref{appendix:problem2}. Finally, 
choosing $T$ as $A(T)=S$ and 
setting 
\[
u(x,t)=a(t)^{-\frac{1}{2}}w\big(x,A(t)\big), \quad t\in [0,T]
\]
with the solution $w$ of \eqref{appendix:problem-translated},
we can find that $u$ is the strong solution of \eqref{appendix:problem}.

Furthermore if $u_0$ and $u_1$ satisfy 
${\rm supp}\,u_0\cup {\rm supp}\,u_1 \subset B(0,r_0)$,
then the corresponding initial data 
$(w_0,w_1)=\big(a(0)^{\frac{1}{2}}u_0, a(0)^{-\frac{1}{2}}u_1+\frac{1}{2}a(0)^{-\frac{3}{2}}a'(0)u_0\big)$ 
for the problem \eqref{appendix:problem-translated} 
are also supported in $B(0,r_0)$. 
The constant speed of propagation for the usual wave equation 
shows that the complete metric subspace 
\[
\mathcal{X}_{S,r_0}=\{\mathcal{W}=(w,{\color{blue}z})\in 
\mathcal{X}_S\;;\;
{\rm supp}\,w(s)\cup {\rm supp}\,z(s) \subset B(0,r_0+s)\}
\]
of $\mathcal{X}_S$ is invariant under the map $\Xi$. This 
implies $\mathcal{W}^\sharp\in \mathcal{X}_{S,r_0}$. 
As a consequence, we additionally 
obtain ${\rm supp}\,u(t)\subset B(0,r_0+A(t))$.
\end{proof}

\subsection*{Acknowedgements}
This work was supported by JSPS KAKENHI Grant Numbers JP24K06811, JP22H00097.

\subsection*{Declarations}
\begin{description}
\item[Data availability] Data sharing not applicable to this article as no datasets were generated or analysed during the current study.
\item[Conflict of interest] 
The authors declare that they have no conflict of interest.
\end{description}

\end{document}